\documentclass[11pt, letterpaper, oneside]{article}

\usepackage{latexsym, amsmath, amssymb, amsfonts, amscd,bm}
\usepackage{amsthm}
\usepackage{t1enc}
\usepackage[mathscr]{eucal}
\usepackage{indentfirst}
\usepackage{graphicx, pb-diagram}
\usepackage{fancyhdr}
\usepackage{fancybox}
\usepackage{paralist}
\usepackage{framed}
\usepackage{color}
\usepackage[all]{xy}
\usepackage{hyperref}
\usepackage{url}
\usepackage{booktabs}
\usepackage{float}

\theoremstyle{plain}
\newtheorem{thm}{Theorem}[section]

\newtheorem{prop}[thm]{Proposition}

\newtheorem{lemma}[thm]{Lemma}
\newtheorem{cor}[thm]{Corollary}

\theoremstyle{definition}
\newtheorem{defi}[thm]{Definition}

\newtheorem{remark}[thm]{Remark}
\newtheorem{ep}[thm]{Example}

\newtheorem{constr}[thm]{Construction}

\newcommand{\RR}{\ensuremath{\mathbb R}}
\newcommand{\g}{\ensuremath{\mathfrak{g}}}

\newcommand{\li}{\ensuremath{L_{\infty}}}

\newcommand{\cL}{\mathcal{L}}

\newcommand{\cC}{\mathcal{C}}
\newcommand{\cA}{\mathcal{A}}

\newcommand{\cP}{\mathcal{P}}

\definecolor{forest}{rgb}{0,0.5,0}

\newcommand{\vs}{\varsigma}

\newcommand{\ham}[1]{\Omega^{#1}_{\mathrm{Ham}}\left(M\right)}
\newcommand{\ip}[1]{\iota_{v_{#1}}}

\newcommand{\alphak}[1]{\alpha_{1} \otimes \cdots \otimes \alpha_{#1}}
\newcommand{\alphadk}[1]{\alpha_{1},\hdots,\alpha_{#1}}

\newcommand{\vk}[1]{v_{\alpha_{1}} \wedge \cdots \wedge v_{\alpha_{#1}}}

\DeclareMathOperator{\ima}{im}

\begin{document}

\title{Conserved quantities on multisymplectic manifolds}

 \author{Leonid Ryvkin\footnote{Fakult\"at f\"ur Mathematik, Ruhr-Universit\"at Bochum, Universit\"atsstr. 150, 44801 Bochum, Germany, \texttt{leonid.ryvkin@rub.de}.},
 Tilmann Wurzbacher\footnote{ Institut \'{E}lie Cartan Lorraine, Universit\'{e} de Lorraine et C.N.R.S., Ile de Saulcy, 57045 Metz, France, 
 \texttt{tilmann.wurzbacher@univ-lorraine.fr}.},
 Marco Zambon\footnote{KU Leuven, Department of Mathematics, Celestijnenlaan 200B box 2400, BE-3001 Leuven, Belgium, \texttt{marco.zambon@kuleuven.be}.}
 }

\date{}

\maketitle

\begin{abstract}
Given a vector field on a manifold $M$, we define a \emph{globally conserved quantity} to be a differential form whose Lie derivative is exact. Integrals of conserved quantities over suitable submanifolds are constant under time evolution, the Kelvin circulation theorem being a well-known special case. 
More generally,   conserved quantities are well-behaved under transgression to spaces of maps into $M$.

We focus on the case of multisymplectic manifolds and Hamiltonian vector fields. We show that in the presence of a Lie group of symmetries admitting a homotopy co-momentum map, one obtains a whole family of globally conserved quantities. This extends a classical result in symplectic geometry. We carry this out in a general setting, considering several variants of the notion of globally conserved quantity.
\end{abstract}

\tableofcontents


\section*{Introduction}
\addcontentsline{toc}{section}{Introduction}

Conserved quantities (or ``conservation laws'') play a large r\^ole in variational problems and related areas, as continuum mechanics. \\

Their mathematical formulation as well as application
is most transparent in the case of classical point mechanics, in its symplectic (or Hamiltonian) presentation.
Given a symplectic manifold $(M,\omega)$ and a Hamiltonian function $H$ in $C^{\infty}(M,\mathbb{R})=
\Omega^0(M)$, a function $f$ on $M$ is a ``conserved quantity'' if the Lie derivative ${\mathcal L}_{v_H}(f)=
-\{H,f\}$ vanishes, where $v_H$ is the Hamiltonian vector field associated to $H$ (i.e., fulfilling $\iota_{v_H}\omega
=-dH$) and $\{\,,\}$ is the Poisson bracket of $(M,\omega)$. If $df$ is different from zero, the dimension of the
phase space $(M,\omega)$ can be reduced by two and the associated Hamilton equation
descends to the reduction. Iterating this process leads  -- at least locally -- to the essentially trivial problem
of solving a Hamilton equation on the real plane with its standard symplectic form. A typical source of conserved quantities is given by the  Noether mechanism, here very simple:
if a finite-dimensional Lie algebra $\mathfrak g$ acts on $(M,\omega)$ with a (co-)momentum map and the Hamiltonian function is $\mathfrak g$-invariant, then the
image of every element of $\mathfrak g$ under the co-momentum is a conserved quantity.\\

The advent of a mathematical rigorous framework for observables and symmetries on a 
multisymplectic manifold $(M,\omega)$ -- i.e., a manifold with a closed, non-degenerate $n{+}1$-form for $n\geq 1$ (cf. \cite{RogersL,RogersPre}) -- raises the question whether the above generalizes from symplectic to multisymplectic geometry.
Accordingly, we consider the set-up of a multisymplectic manifold $(M,\omega)$ and a
``Hamiltonian form'' $H\in \Omega^{n-1}(M)$, allowing for a vector field $v_H$ such that $\iota_{v_H}\omega=-dH$. We call a differential form $\alpha \in \Omega^{\bullet}(M)$
``strictly conserved by $H$ (or under $v_H$)'' if $\mathcal{L}_{v_H}\alpha=0$.
Working with forms rather than functions, we immediately have two natural weakened notions:
``global conservation'' resp. ``local conservation'' in case $\mathcal{L}_{v_H}\alpha$ is exact resp. closed.
Since conserved quantities are typically considered in integrated form, it is often enough that a ``quantity is
preserved up to a total divergence'', which corresponds to these two weakened notions, that are less interesting in the
symplectic case. Among these three kinds of conserved quantities, the one we consider most useful are the globally conserved quantities.\\

Our main goal is to understand to which extent a homotopy co-momentum associated to a multisymplectic action of a finite-dimensional Lie algebra $\mathfrak g$ (cf. \cite{FRZ}, \cite{LeonidTillmannComoment}) furnishes conserved quantities if $\mathfrak g$ keeps 
the Hamiltonian form $H$ invariant. Given the more involved algebraic 
structure of the observables (and of the homotopy co-momentum), we find the following as the ``correct'' generalization of the above conservation law on symplectic manifolds (see Prop. \ref{prop:conserved}):\\

{\bf Proposition.}  Let $(M,\omega)$ be a multisymplectic manifold, $H\in \ham{n-1}$ and $(f)$ a homotopy co-momentum map for an infinitesimal action $\mathfrak g\to \mathfrak X(M)$ which leaves $H$ invariant. Let $p$ be a $k$-cycle in 
the complex defining the Lie algebra homology of $\mathfrak g$. Then $f_k(p)$ is a globally conserved quantity.\\

In fact, denoting the cycles of Lie algebra homology by $Z_k(\mathfrak g)$ and the boundaries by $B_k(\mathfrak g)$, we have the following table, {whose last column reflects the above proposition} (see Definition \ref{def:pres} for the three notions of
preservation of H by the action of the 
Lie algebra $\g$): 
\begin{table}[H]
 \centering
 \begin{tabular}{ l | c| c | c }
 \toprule
 & $H$ locally $\g$-preserved & $H$ globally $\g$-preserved & $H$ strictly $\g$-preserved \\
 \hline
 $f_k(Z_k(\g))$ & locally conserved & locally conserved & globally conserved \\ 
 $f_k(B_k(\g))$ & globally conserved & globally conserved & globally conserved \\ 
 \bottomrule
\end{tabular}
 \end{table}
We underline that all implications in the table are sharp, as we show by explicit examples. {Further, our results hold even relaxing the assumption that $\omega$ be multisymplectic, allowing $\omega$ to be any closed $n+1$-form.}\\

In discussing conserved quantities (in the flavors strong/global/local), it turns out to be useful to first work in the 
more general situation of a manifold $M$ together with a vector field $v$ and to discuss differential forms ``preserved 
by this continuous dynamical system''. In this general context, we associate ``integral invariants'' to conserved quantities 
by integrating the conserved forms over manifolds that are smoothly mapped to $M$. We obtain a very general form of 
Kelvin's classical circulation theorem, that should be of use in continuum mechanics beyond the case of isentropic,
incompressible fluids. More precisely we have (compare with Proposition \ref{prop:unchanged}):\\

{\bf Proposition.}
Let $\Sigma$ be a compact, oriented $d$-dimensional manifold (without boundary), $v$ a vector field on $M$ with flow $\phi_t$,
and $\sigma_0\colon \Sigma\to M$ a smooth map. Consider $\sigma_t:=\phi_t\circ \sigma_0\colon \Sigma\to M$.
If $\alpha\in \Omega^d(M)$ is a differential form, 
then the number
$$\int_{\Sigma}(\sigma_t)^*\alpha$$ is independent of the time parameter $t$ if 
$\alpha$ is globally conserved by $v$.\\ 

The above proposition makes apparent, in a geometric way, the usefulness of conserved quantities.\\

Let us now describe the content of the different sections in more detail. 
In \S \ref{sec:consmulti} we introduce   multisymplectic manifolds and define the various notions of conserved quantities.
The heart of this note is 
\S \ref{sec:conshomo}: given an action on multisymplectic manifold that preserves (in one of the ways we make precise) a Hamiltonian form, we show that certain components of
the homotopy co-momentum map are conserved quantities. Further, in \S \ref{subsec:homol} we provide an alternative, homological approach to prove these statements. 
In \S \ref{sec:exconstr} we explain how the set-up needed in the previous section arises naturally,   we make  remarks on conserved quantities, and provide in \S \ref{subsec:magnetic} an example, announced in \S \ref{sec:strict}, exhibiting a globally conserved quantity (with respect to a strictly $H$-preserving action) that is not strictly conserved. Finally,  
\S \ref{sec:applications} is devoted to applications: we present a general version of Kelvin's circulation theorem, and more generally we show that conserved quantities on a manifold $M$ induce conserved quantities on  spaces of maps into $M$.\\

For a discussion of the relation between conserved quantities in multisymplectic geometry on one side and classical field theory on the other side,
we refer to Schreiber \cite[\S 1.2.11]{UrsLong}.\\

{\bf Acknowledgements:} M.Z. thanks Chris Rogers for constructive comments, especially during the first phase of the project, and Daniel Peralta Salas for making him aware of the Kelvin Circulation Theorem. M.Z. was partially supported by grants MTM2011-22612 and ICMAT Severo Ochoa SEV-2011-0087 (Spain), 
Pesquisador Visitante Especial grant 88881.030367/2013-01 (CAPES/Brazil), by IAP Dygest (Belgium), by
the long term structural funding Methusalem grant of the Flemish Government. We thank the IECL Metz-Nancy (UMR 7502, France), KU Leuven and Ruhr-University Bochum (via its Research School Plus, funded by Germany's Excellence Initiative [DFG GSC 98/3]) for making 
possible mutual visits of the authors.

\section{Conserved quantities in multisymplectic geometry}\label{sec:consmulti}
The purpose of this section is to address conserved quantities associated to a vector field on a manifold. We will be mainly interested in the case that the manifold carries a multisymplectic structure and the vector field is Hamiltonian. We 
introduce these notions in \S \ref{subsec:introm}  and \S \ref{subsec:def-conserved} , and display some algebraic properties of conserved quantities in \S
\ref{subsec:alg}.

\subsection{Multisymplectic manifolds}\label{subsec:introm}
\begin{defi}
A manifold $M$ equipped with a closed $n{+}1$-form $\omega\in\Omega^{n+1}(M)$ is called a \emph{pre-$n$-plectic manifold}. It is called an \emph{$n$-plectic} or \emph{multisymplectic manifold} if the following map is injective for all $p\in M$,
\[
{T_pM\to \Lambda^nT^*_pM,\;\; v\mapsto \iota_v\omega_p}.
\]
\end{defi}

 \begin{defi} \label{Hamiltonian}
Let $(M,\omega)$ be a pre-$n$-plectic manifold. An $(n{-}1)$-form $\alpha$
is called \emph{Hamiltonian} if there exists a vector field $v_\alpha \in \mathfrak{X}(M)$ such that
\[
d\alpha= -\ip{\alpha} \omega.
\]
We say that $v_\alpha$ is a \emph{Hamiltonian vector field} for $\alpha$. In the $n$-plectic case $v_\alpha$ is unique. The set of Hamiltonian $(n{-}1)$-forms is denoted as $\ham{n-1}$.
\end{defi}

\begin{remark}
Observe that if $v_\alpha$ is a Hamiltonian vector field corresponding to $\alpha$, then $\cL_{v_\alpha}\omega=0$ by Cartan's formula.
\end{remark}

The following $L_{\infty}$-algebra was constructed for $n$-plectic manifolds in \cite[Thm. 5.2]{RogersL} and generalized to the pre-$n$-plectic case in \cite[Thm. 6.7]{HDirac} 
\begin{defi}\label{def:lichris} 
Given a pre-$n$-plectic manifold $(M,\omega)$, the \emph{Lie $n$-algebra of observables}
$L_{\infty}(M,\omega)=(L,\{l_{k} \})$ is the graded vector space given by 
\[
L_{i} =
\begin{cases}
\ham{n-1} & i=0,\\
\Omega^{n-1-i}(M) & 0 < i \leq n-1,
\end{cases}
\]
together with the maps $\left \{l_{k} \colon L^{\otimes k} \to L~|~ 1
 \leq k\leq n{+}1 \right\}$ given by
\[ 
l_{1}(\alpha)=d\alpha \qquad \text{if } \deg{\alpha}>0,
\]
$l_1(\alpha)=0$ for $\deg{\alpha}=0$, and for all $k>1$ 
\[
l_{k}(\alphadk{k}) =
\begin{cases}
0 & \text{if $\deg{\alphak{k}} > 0$}, \\
\vs(k) \iota(\vk{k}) \omega & \text{if
 $\deg{\alphak{k}}=0$}, 
 \end{cases}
\]
 where $v_{\alpha_{i}}$ is a Hamiltonian vector field
associated to $\alpha_{i} \in \ham{n-1}$ and
 $\varsigma(k)=-(-1)^{k(k+1)/2}$.
Here the contraction with multivector fields is defined by $\iota(\vk{k})=\iota_{v_{\alpha_k}}\dots \iota_{v_{\alpha_1}}$.
\end{defi}

\subsection{Conserved quantities}\label{subsec:def-conserved}
\begin{framed}
\begin{defi}\label{def:consq}
Let $M$ be a manifold and $v$ a vector field on $M$. A form $\alpha \in \Omega^\bullet(M)$ is called a 
\begin{compactenum}[(a)]
	\item \emph{locally conserved quantity} if $\cL_{v}\alpha$ is a closed form,
	\item \emph{globally conserved quantity} if $\cL_{v}\alpha$ is an exact form,
	\item \emph{strictly conserved quantity} if $\cL_{v}\alpha=0$.
\end{compactenum}
We denote the graded vector spaces of those quantities by $\cC_{loc}(v)$ resp. $\cC(v)$ and $\cC_{str}(v)$.
\end{defi}
\end{framed}

The following inclusions follow directly from Cartan's formula.

\begin{lemma} \label{conserved-properties}Let $M$ be a manifold and $v$ a vector field on $M$. Then we have
\begin{compactenum}[(i)]
\item $\cC_{str}(v)\subset \cC(v)\subset \cC_{loc}(v)$,
\item $\Omega^\bullet_{cl}(M)\subset \cC(v)$,
\item $d( \cC_{loc}(v))\subset \cC_{str}(v)$.
\end{compactenum}
\end{lemma}

We will be especially interested in the case  where $(M,\omega)$ is pre-$n$-plectic and $v$ preserves $\omega$. In this case additional results hold.

\begin{lemma} \label{conserved-properties2}Let $(M,\omega)$ be a pre-$n$-plectic manifold and $v$ a vector field on $M$ such that $\mathcal L_v\omega=0$. Then we have
\begin{compactenum}[(i)]
\item $\alpha\in \ham{n-1}$ is locally conserved by $v$ if and only if $\iota_{[v_\alpha,v]}\omega=0$, for some (or equivalently for {every}) Hamiltonian vector field $v_\alpha$ for $\alpha$.
\end{compactenum}
If moreover $v=v_H$ is a Hamiltonian vector field for $H\in \ham{n-1}$, then 
\begin{compactenum}[(i)]\setcounter{enumi}{1}
\item $\alpha\in \ham{n-1}$ is locally conserved by $v_H$ if and only if $\cL_{v_\alpha}H$ is closed {for some (or equivalently for every)} Hamiltonian vector field $v_\alpha$ for $\alpha$.
\item $\alpha\in \ham{n-1}$ is {globally} conserved by $v_H$ if and only if $\cL_{v_\alpha}H$ is exact {for some (or equivalently for every)} Hamiltonian vector field $v_\alpha$ for $\alpha$.
\item $H\in \cC(v_H)$.

\end{compactenum}
\end{lemma}

\begin{proof}
 Assertion (i) follows from the identity $\mathcal L_X\circ \iota_Y=\iota_Y \circ\mathcal L_X+\iota_{[X,Y]}$ applied to $\omega$. Assertions (ii) - (iv) follow from Cartan's formula.
\end{proof}

\begin{remark}
We observe that the closeness resp. exactness of $\cL_{v_\alpha}H$ is equivalent to the closeness resp. exactness of $l_2(\alpha,H)$.
\end{remark}

As the following example illustrates, even in the $n$-plectic case, in general $\cL_{v_H}H\neq 0$.

\begin{ep}
Let $M=\mathbb R^3$, $\omega=dx\wedge dy\wedge dz$ and $H=xdy+zdz$. Then $v_H=-\frac{\partial}{\partial z}$, so $\iota_{v_H}H=-z$ and $\cL_{v_H}H=-dz$.
\end{ep}
\begin{remark}
In the symplectic (i.e. the 1-plectic) case with $H\in C^\infty(M)=\Omega^0_{Ham}(M)$ we have the following statements for $f\in C^\infty(M)$ and $v=v_H$:
\begin{compactenum}[(1)]
\item $f $ is \emph{{globally} conserved} if and only if $f$ is \emph{strictly conserved} and this is the case, if and only if $\{H,f\}=0$,
\item $f$ is \emph{locally conserved} if and only if $\{H,f\}$ is locally constant.
\end{compactenum}

As the following example shows, in the symplectic situation local conservedness does not suffice to formulate a ``conservation law''.
\end{remark}

\begin{ep}
Let $M=\mathbb R^2$ with coordinates $q,p$, $\omega=dp\wedge dq$ and $H=p$. Taking $f=q$, the Hamiltonian vector field is given by $v_H=\frac{\partial}{\partial q}$ and thus $\cL_{v_H}f=1$ i.e. $f$ is locally but not globally conserved. Then for any integral curve $\gamma(t)=(q_0+(t-t_0),p_0)$ of $v_H$ we have $f(\gamma(t))=f(\gamma(t_0))+(t-t_0)$, i.e., $f$ is not a constant of motion.
\end{ep}

\subsection{The algebraic structure of conserved quantities}\label{subsec:alg}

Throughout this subsection $M$ will denote a manifold and $v$ a vector field on $M$. We present here elementary methods to construct new conserved quantities from known ones. 

\begin{lemma}
The space $\cC_{str}(v)$ is a graded subalgebra of $\Omega^\bullet (M)$.
\end{lemma}
As the following example illustrates the spaces $\cC(v)$ and $\cC_{loc}(v)$, unlike $\cC_{str}(v)$, are not closed under wedge-multiplication.
\begin{ep}
Let $M=\mathbb R^3$, $\omega=dx\wedge dy \wedge dz$ and $H=-xdy$. We observe that $dH=-dx\wedge dy$ and consequently $v_H=\frac{\partial}{\partial z}$. We set $\alpha=zdx$ and $\beta=zdy$. Then $\cL_{v_H}\alpha=dx$ and $\cL_{v_H}\beta=dy$ are exact but $\cL_{v_H}(\alpha\wedge \beta)=2zdx\wedge dy$ is not even closed.
\end{ep}

However stability under multiplication with elements from the following graded-commutative subalgebra of $\Omega^\bullet(M)$ is assured:

$$\cA(v):=\{\beta\in \Omega(M)~|~d\beta=0 \text{ and }\cL_{v}\beta=0\}\subset \cC_{str}(v).$$

\begin{lemma}\label{lem:module}
The spaces $\cC(v)$ and $\cC_{loc}(v)$ are graded modules over $\cA(v)$.
\end{lemma}
\begin{proof}
We prove the statement for $\cC(v)$, the proof for $\cC_{loc}(v)$ being identical. Let $\alpha\in \cC(v)$ (that is, there is a form $\gamma$ with $\cL_{v}\alpha=d\gamma$) and $\beta\in \cA(v)$. Then
$$\cL_{v}(\alpha\wedge \beta)=\cL_{v}\alpha\wedge \beta+
\alpha\wedge \cL_{v}\beta=d\gamma\wedge \beta=d(\gamma\wedge \beta).$$
\end{proof}

Again, more can be said if $(M,\omega)$ is pre-$n$-plectic and $v$ preserves $\omega$.

\begin{prop}\label{conserved-subalgebra} Let $(M,\omega)$ be pre-$n$-plectic and $v$ a vector field satisfying $\mathcal L_v\omega=0$.
The graded vector spaces $$L_\infty(M,\omega)\cap \cC_{loc}(v),\;\;L_\infty(M,\omega)\cap \cC(v)\text{ and }L_\infty(M,\omega)\cap \cC_{str}(v)$$ are $L_\infty$-subalgebras of $L_\infty(M,\omega)$. Moreover $\mathcal L_{v} \left(l_k(\beta_1,...,\beta_k)\right)=0$ for $k\geq 1$ and $\beta_1,...,\beta_k\in L_\infty(M,\omega)\cap \cC_{loc}(v)$.
\end{prop}
\begin{proof}
We claim that brackets of locally conserved quantities in $L_\infty(M,\omega)$ are strictly conserved. The only bracket which is nontrivial on components other than $\Omega^{n-1}_{Ham}(M)$ is $l_1=d$. It follows from part (iii) of Lemma \ref{conserved-properties} that $l_1=d$ applied to a locally conserved quantity is strictly conserved. Now for $k\ge 2$ consider $\beta_1,...,\beta_k\in \Omega^{n-1}_{Ham}(M)$, such that $\mathcal L_{v}\beta_i$ is closed for all $i$. We want to show that
\[
\mathcal L_{v} \left(l_k(\beta_1,...,\beta_k)\right)=0.
\]
As $ l_k(\beta_1,...,\beta_k)=\pm\iota(v_{\beta_1}\wedge\dots\wedge v_{\beta_k})\omega$, for any collection of Hamiltonian vector fields $\{v_{\beta_i}\}$ for $\{\beta_i\}$ this is equivalent to showing
\[
\mathcal L_{v} \iota_{v_{\beta_k}}...\iota_{v_{\beta_1}}\omega=0.
\] 
Using the identity $\mathcal L_X\circ \iota_Y=\iota_Y \circ\mathcal L_X+\iota_{[X,Y]}$ we can move $\mathcal L_{v}$ past the $\iota_{v_{\beta_i}}$ since $\iota_{[v,v_{\beta_i}]}\omega=0$ by part (i) of Lemma \ref{conserved-properties2}. We find 
\[
\mathcal L_{v}\iota_{v_{\beta_k}}...\iota_{v_{\beta_1}}\omega=\iota_{v_{\beta_k}}...\iota_{v_{\beta_1}}\mathcal L_{v}\omega=0,
\] 
proving our claim.
\end{proof}

\section{Conserved quantities from homotopy co-momentum maps}\label{sec:conshomo}
In this section we consider (infinitesimal) actions. 
More precisely, $(M,\omega)$ will always denote a pre-$n$-plectic manifold and G a Lie group (resp. $\mathfrak g$ a Lie algebra) acting on $M$.

Given a Hamiltonian form $H$ on $M$, we define three notions of ``preservedness'' of $H$ with respect to the action, see Definition \ref{def:pres}. Suitable conserved quantities are constructed from co-momentum maps for each of these three notions of preservedness,
respectively in \S \ref{subsec:locpres},
\S \ref{subsec:glopres} and
\S \ref{sec:strict}. Finally, in
\S \ref{subsec:homol} we propose a less intuitive homological approach that has the advantage of being rather concise.

\subsection{Actions on multisymplectic manifolds}\label{subsec:action}

\begin{defi}\label{def:part} Let $(M,\omega)$ be a pre-$n$-plectic manifold. A right action $\vartheta$ of a Lie group $G$ on $M$ is called \emph{multisymplectic} if $\vartheta^*_g\omega=\omega$ for all $g\in G$, where $\vartheta_g=\vartheta(\cdot,g)$. An infinitesimal right action of a Lie algebra $\mathfrak g$ on $M$, i.e. a Lie algebra homomorphism $\mathfrak g\to \mathfrak X(M), x\mapsto v_x$, is called \emph{multisymplectic} if $\cL_{v_x}\omega=0$ for all $x\in\mathfrak g$. For a connected Lie group $G$, a right action $\vartheta$ is multisymplectic if{f} the corresponding infinitesimal right action (given by $x\mapsto v_x$ where $v_x(m)=\frac{d}{dt}|_0 \vartheta(m,exp(tx))$ at all $m\in M$) is multisymplectic.
\end{defi}

A multisymplectic infinitesimal action is thus a Lie algebra homomorphism from $\mathfrak g$ to $\mathfrak X(M,\omega)=\{X\in\mathfrak X(M)| \cL_X\omega=0\}$. One may ask, whether such an action admits an ``$\li$-lift'' to $\li(M,\omega)$. For an explicit description of the equations fulfilled by such a lift, the following definition is useful.

\begin{defi}\label{def:differential} Let $\mathfrak g$ be a Lie algebra. We define the \emph{Lie algebra homology differential} 
$\partial$ by setting 
$$\partial_k=\partial|_{\Lambda^k\mathfrak g} \colon \Lambda^k\g \to \Lambda^{k-1}\g, x_1\wedge\dots\wedge x_k
\mapsto \sum_{1 \leq i < j \leq
 k} (-1)^{i+j} [x_{i},x_{j}] \wedge x_{1} \wedge \cdots
 \wedge \hat{x}_{i} \wedge \cdots \wedge \hat{x}_{j} \wedge \cdots \wedge x_{k},$$ 
for $k\geq 1$. We put $\Lambda^{-1}\mathfrak g=\{0\}$ and $\partial_0$ to be the zero map.
\end{defi}

\begin{defi}An \emph{$\li$-morphism} $(f)$ from $\mathfrak g$ to $\li(M,\omega)$ is a collection of maps $(f)=\{f_i:\Lambda^i\mathfrak g\to \Omega^{n-i}(M)| 1\leq i\leq n\}$ satisfying $\ima(f_1)\subset \ham{n-1}$ and the equation
 \begin{equation}\label{main}
 -f_{k-1}(\partial(p))=df_k(p)+ \vs(k)\iota(v_p)\omega
\end{equation}
 for all $k=1,\ldots,n+1$ and $p\in \Lambda^k\g$
 (setting $f_0$ and $f_{n+1}$ to be zero). Here we use the short hand notation $v_p:=v_{x_1}\wedge\dots\wedge v_{x_k}$ whenever $p= {x_1}\wedge\dots\wedge{x_k}$ for $x_i\in\g$.
\end{defi} 
\begin{remark}
This is, of course, the general definition of an $L_\infty$-algebra morphism specialized to the case at hand.
\end{remark}

We recall from \cite[\S 5]{FRZ} the higher analogue of momentum map in symplectic geometry:
 \begin{defi}
A \emph{(homotopy) co-momentum map} for {a multisymplectic infinitesimal action} $v:\mathfrak g\to \mathfrak X(M)$ on $(M,\omega)$ is an $\li$-morphism 
$(f)\colon \g \to L_{\infty}(M,\omega)$ such 
that for all $x\in \g$
\begin{equation}\label{eq:mom}
d(f_1(x))=-\iota_{v_x}\omega.
\end{equation}
\end{defi}

\begin{remark} $ $
\begin{compactenum}[(1)]
\item In \cite{FLRZ} and \cite{LeonidTillmannComoment} equations \eqref{main} and \eqref{eq:mom} are interpreted as a coboundary condition on a certain chain complex.
\item 
A co-momentum map $(f)$ is \emph{$G$-equivariant} if the components $f_i\colon \Lambda^i\mathfrak g\to \Omega^{n-i}(M)$ are equivariant for all $i\in\{1,...,n\}$. When $G$ is connected, this can be expressed infinitesimally:
$\forall q\in \Lambda^{i}\mathfrak g$ and for all $x\in \mathfrak g=T_eG$ the equality $\cL_{v_x}(f_i(q))=f_i([x,q])$ holds. Here $[x,\cdot]$ is $ad(x)$ acting  on $\Lambda^\bullet \mathfrak g$.
\end{compactenum}
\end{remark}

Now we turn to infinitesimal actions preserving a Hamiltonian $n{-}1$-form $H$ on a pre-$n$-plectic manifold $(M,\omega)$. As in the case of the conserved quantities, one has to distinguish to which extent the action preserves the Hamiltonian form.

\begin{framed}
\begin{defi}\label{def:pres} Let $\mathfrak g\to \mathfrak X(M,\omega), x\mapsto v_x$ be an infinitesimal action. It is called
\begin{compactenum}[(a)]
	\item \emph{locally $H$-preserving} if $\cL_{v_x}H$ is closed for all $x\in \mathfrak g$.
	\item \emph{globally $H$-preserving} if $\cL_{v_x}H$ is exact for all $x\in \mathfrak g$.
	\item \emph{strictly $H$-preserving} if $\cL_{v_x}H=0$ for all $x\in \mathfrak g$.
\end{compactenum}
\end{defi}
\end{framed}
\begin{remark}
Usually a differential form would be called ``preserved by an infinitesimal action'' if condition (c) is fulfilled.
\end{remark}

In the following we will investigate the conserved quantities arising from co-momentum maps separately for these three cases.

\subsection{Conserved quantities from locally $H$-preserving actions}\label{subsec:locpres}

In this subsection we assume that $(M,\omega)$ is a pre-$n$-plectic manifold, $H\in\ham{n-1}$ and that $(f):\mathfrak g\to \li(M,\omega)$ is the co-momentum of a \emph{locally} $H$-preserving infinitesimal action $\mathfrak g\to \mathfrak X(M,\omega),x\mapsto v_x$.\newline

By the definition of a co-momentum map, the generator of the infinitesimal action associated to $x$ in $\g$ is a Hamiltonian vector field of $f_1(x)$. As earlier,  for $p=x_1\wedge ...\wedge x_k\in\Lambda^k\mathfrak g$ we write $v_p:=v_{x_1}\wedge ...\wedge v_{x_k}$ and $\iota({v_p})=\iota_{v_{x_k}}...\iota_{v_{x_1}}$.

\begin{lemma}\label{lemma:one} Let $(f)=\{ f_i | 1\leq i\leq n \} $ be a co-momentum for $ v:\mathfrak g\to \mathfrak X (M,\omega)$ and $H\in \ham{n-1}$. Then for any Hamiltonian vector field $v_H$ of $H$ we have 
\begin{compactenum} [(i)]
	\item $f_1(x)\in \cC_{loc}(v_H)$ for all $x\in\mathfrak g$,
	\item $\iota_{[v_H,v_x]}\omega=0$ for all $x\in\mathfrak g$,
	\item $\iota(v_p)\omega\in \cC_{str}(v_H)$ for all $p\in \Lambda^k\mathfrak g$.
\end{compactenum} \label{Lemma:loc-pres} 
\end{lemma}

\begin{proof}
(i) follows from Lemma \ref{conserved-properties2} (ii) and (ii) from Lemma \ref{conserved-properties2} (i). Further, (iii) follows upon recalling that $[\cL_v,\iota_w]=\iota_{[v,w]}$ and part (ii):
$$
\cL_{v_H}(\iota(v_p)\omega)=\cL_{v_H}\iota_{v_{x_k}}...\iota_{v_{x_1}}\omega=-\iota_{v_{x_k}}\cL_{v_H}...\iota_{v_{x_1}}\omega=...=\pm\iota(v_p)(\cL_{v_H}\omega)=0.
$$
\end{proof}

It turns out, that certain subspaces of the image of the higher components of the co-momentum map constitute locally conserved quantities. To specify this we recall the definition of Lie algebra homology.

\begin{defi}\label{def:cob}
Let $\mathfrak g$ be a Lie algebra, $k\geq 1$ and $\partial_k$ the $k$-th Lie algebra homology differential.
We define
\begin{compactenum}[(a)]
	\item the \emph{cycles} $Z_k(\mathfrak g)=\ker(\partial_k)\subset \Lambda^k\mathfrak g$,
	\item the \emph{boundaries} $B_k(\mathfrak g)=\ima(\partial_{k+1})\subset \Lambda^k\mathfrak g$ and
	\item the $k$-th \emph{Lie algebra homology space} $H_k(\mathfrak g)=\frac{Z_k(\mathfrak g)}{B_k(\mathfrak g)}$.
\end{compactenum}
\end{defi}
\begin{remark}The space $Z_k(\mathfrak g)$ is denoted by $\cP_{\g,k}$ and called the \emph{k-th Lie kernel} of $\g$ in \cite{MadsenSwannClosed}.
\end{remark}

\begin{prop}\label{prop:loc-conserved}
Let $p\in Z_k(\mathfrak g)$. Then $f_k(p)$ is locally conserved by any Hamiltonian vector field $v_H$ of $H$.
\end{prop}

\begin{proof}
The case $k=1$ is part (i) of Lemma \ref{Lemma:loc-pres}. Assume now $k>1$. We have to show that $\mathcal L_{v_H}f_k(p)$ is closed. We have
\begin{equation*}\label{eq:f1p}
d\mathcal L_{v_H}f_k(p)=\mathcal L_{v_H}df_k(p)=-\vs(k) \mathcal L_{v_H} \iota(v_p)\omega=0,
\end{equation*}
where the first equality holds because the Lie derivative commutes with the exterior derivative, the second one, because of Equation (\ref{main}), and the last one because of Lemma \ref{Lemma:loc-pres} (iii).
\end{proof}

Proposition \ref{prop:loc-conserved} states that $\cL_{v_H}f_k(p)$ is a closed $(n{-}k)$-form, hence we obtain:

\begin{cor}
Let $p\in Z_k(\mathfrak g)$.  If $H^{n-k}_{dR}(M)$ is zero, then $f_k(p)$ is globally conserved  by any Hamiltonian vector field $v_H$ of $H$.
\end{cor}
\begin{prop}\label{prop:ex-loc-conserved}
If $p\in B_k(\mathfrak g)\subset Z_k(\mathfrak g)$, then $f_k(p)$ is globally conserved  by any Hamiltonian vector field $v_H$ of $H$.
\end{prop}

\begin{proof}
Let $q$ be a potential for $p$, i.e. $\partial_{k+1}q=p$. Then
\begin{align*}
\cL_{v_H}(f_k(p))=\cL_{v_H}(f_k(\partial q))&=\cL_{v_H}(-df_{k+1}(q) -\vs(k+1)\iota(v_q)\omega)\\&=-d\cL_{v_H}f_{k+1}(q) -\vs(k+1)\cL_{v_H}\iota(v_q)\omega),
\end{align*}
using Equation (\ref{main}). The statement then follows by Lemma \ref{Lemma:loc-pres} (iii).
\end{proof}

The following example shows sharpness of the statement of Proposition \ref{prop:loc-conserved}, i.e. for $p\in \Lambda^k\mathfrak g$ the condition $\partial p=0$, in general, does not imply that $f_k(p)$ is globally conserved.

\begin{ep}\label{local-example}
Let $M=\mathbb R^3$, $\omega=dx\wedge dy \wedge dz$ and $H=-xdy$. We already observed $dH=-dx\wedge dy$ and $v_H=\frac{\partial}{\partial z}$. We consider the two-dimensional abelian Lie algebra $\mathfrak g=\langle a,b\rangle_\RR$ and the homomorphism $v:\mathfrak g\to \mathfrak X(M)$ given by $v_a=\frac{\partial}{\partial x}$ and $v_b=\frac{\partial}{\partial y}$. We have that $\cL_{v_a}H=-dy$ is exact and $\cL_{v_b}H=0$.
 We construct a co-momentum map for this action by $f_1(a)=-ydz$, $f_1(b)=xdz$ and $f_2(a\wedge b)=-z$. 
Then $a\wedge b\in Z_2(\mathfrak g)$ is a cycle and $-z$ is locally conserved, as predicted by Proposition \ref{prop:loc-conserved}, but not globally:
$$\mathcal L_{v_H}(-z)=\iota_{\frac{\partial}{\partial z}}d(-z)=-1\neq 0 .$$
\end{ep}

\subsection{Conserved quantities from globally $H$-preserving actions}\label{subsec:glopres}

In this subsection we assume that $(M,\omega)$ is a pre-$n$-plectic manifold, $H\in\ham{n-1}$ and that $(f):\mathfrak g\to \li(M,\omega)$ is the co-momentum of a \emph{globally} $H$-preserving infinitesimal action 
$\mathfrak g\to \mathfrak X(M,\omega),x\mapsto v_x$.\\

As Example \ref{local-example} indicates, no significant improvements of the above results are to be expected upon passing from locally to globally $H$-preserving actions. There is only a slight improvement of Lemma \ref{Lemma:loc-pres} (i) with essentially the same proof:

\begin{lemma}Let $(f)=\{f_i|1\leq i\leq n\}$ be a co-momentum for $v:\mathfrak g\to \mathfrak X(M)$. Then $f_1(x)\in \cC(v_H)$ for all $x\in\mathfrak g$ and for any Hamiltonian vector field $v_H$ of $H$. \label{lem:easy}
\end{lemma}

\begin{remark}
Notice that Lemma \ref{Lemma:loc-pres} and Lemma \ref{lem:easy} hold for any element $\ham{n-1}$ whose Hamiltonian vector field is $v_x$.
\end{remark}

\subsection{Conserved quantities from strictly $H$-preserving actions} \label{sec:strict}

In this subsection we assume that $(M,\omega)$ is a pre-$n$-plectic manifold, $H\in\ham{n-1}$ and that $(f):\mathfrak g\to \li(M,\omega)$ is the co-momentum of a \emph{strictly} $H$-preserving infinitesimal action $\mathfrak g\to \mathfrak X(M,\omega),x\mapsto v_x$.
The assumption that the action is strictly $H$-preserving can be easily realized 
if there is a connected compact Lie group $G$ with a smooth action on $M$ whose differential is the infinitesimal action of $\mathfrak g$ on $M$, see Lemma \ref{lem:strictp}.\\

To prove a stronger result than Proposition \ref{prop:loc-conserved} in this situation we need the following observation (cf., e.g., \cite[Lemma 3.4]{MadsenSwannClosed}).

\begin{lemma}\label{tech_lemma}
Let $M$ be a manifold and let $\Omega$ be a not necessarily closed differential form on $M$. For all $m \geq 1$ and all vector fields $v_1,\dots,v_m$ in the Lie algebra $\mathfrak X(M)$ we have:
\begin{align*} 
(-1)^{m}d \iota(v_{1} \wedge\cdots \wedge v_{m}) \Omega &= 
\iota(\partial(v_{1}\wedge\ldots\wedge v_{m}))\Omega
+\sum_{1=1}^{m} (-1)^{i} \iota( v_{1} \wedge \cdots
 \wedge \hat{v}_{i} \wedge \cdots \wedge {v}_{m})\cL_{v_i}\Omega\\
&+ \iota( v_{1} \wedge \cdots
 \wedge {v}_{m}) d\Omega.
\end{align*}
\end{lemma}

\begin{defi}\label{multidef}
Given a differential form $\Omega\in \Omega^\bullet(M)$ and a multi-vector field $Y\in \Gamma(\Lambda^m TM)$, the \emph{Lie derivative of $\Omega$ along $Y$} is defined as a graded commutator,
by $\cL_Y\Omega:=d\iota_Y\Omega-(-1)^m\iota_Yd\Omega$.
\end{defi}
\begin{remark}\label{rem:Lieder}
This definition allows to combine the first and last term in the above formula into a Lie derivative.
Hence the above formula can be written 
$\cL_{v_{1} \wedge \cdots
\wedge {v}_{m}}\Omega 
=(-1)^m[
\iota(\partial(v_{1}\wedge\ldots\wedge v_{m}))\Omega+\sum_{1=1}^{m} (-1)^{i} \iota( v_{1} \wedge \cdots
 \wedge \hat{v}_{i} \wedge \cdots \wedge {v}_{m})\cL_{v_i}\Omega]$.
\end{remark}

\begin{prop}\label{prop:conserved}
Let $p\in Z_k(\mathfrak g)$. Then
$f_k(p)$ is a globally conserved quantity.
\end{prop}
\begin{proof} We have
\begin{align*}
\iota_{v_H} df_k(p)
=-\vs(k)\iota_{v_H}\iota(v_p)\omega
=(-1)^k\vs(k)\iota(v_p)dH
=\vs(k)d(\iota(v_p)H)
\end{align*}
where we used 
Equation \eqref{main} in the first equality, and in the last equality Lemma \ref{tech_lemma} applied to the form $H$, as well as the $\mathfrak g$-invariance of $H$ and once more the assumption $p\in Z_k(\g)$.

Therefore we conclude from Cartan's formula that
$$\cL_{v_H} f_k(p)=d\big(\iota_{v_H} f_k(p)+
\vs(k)\iota(v_p)H\big).$$
\end{proof}

\begin{remark}
In the symplectic case, a homotopy co-momentum map boils down to its first component, $f_1:\mathfrak{g}\to
 \Omega^0(M)$, a classical co-momentum map. Upon observing that $Z_1(\mathfrak g)=\mathfrak g$, the preceding proposition
then reduces to the well-known obvious but important fact that if a Hamiltonian function $H$ is $\mathfrak g$-invariant, then
for all $x$ in $\mathfrak g$ we have that $\{f_1(x),H\}=\cL_{v_H} f_1(x)=0$.
\end{remark}

In particular, by Proposition \ref{prop:conserved},
$f_1(x)$  is a globally conserved quantity for all $x\in \mathfrak g$.
Even in the case at hand of strictly $H$-preserving actions, $f_1(x)$ is not strictly conserved in general, as the following example shows.
\begin{ep}
Let $M=\mathbb R^3$, $\omega=dx\wedge dy\wedge dz$ and $H=-xdy$. Then $v_H=\frac{\partial}{\partial z}$. Furthermore we consider $\alpha=zdx$ and the $\mathbb R$-action given by $\mathfrak g=\mathbb R\to \mathfrak X(M)$, $1\mapsto v_\alpha=-\frac{\partial}{\partial y}$. This action clearly admits a co-momentum map determined by $f_1(1)=\alpha$. Then $\cL_{v_\alpha}H=0$ but $\cL_{v_H}\alpha=dx\neq 0$.
\end{ep}

More is true: even if one assumes that $x\in B_1(\g)$ is a boundary, $f_1(x)$ is still not strictly {conserved}  in general, as Example \ref{magnetic:ep} below shows.

Specializing $k$ to $n$ in Proposition \ref{prop:conserved}, we obtain scalar functions 
$f_n(x_1,\dots,x_n)$ on $M$. Assembling these functions we obtain a map $M\to Z_n(\g)^*$, very similar to the multi-momentum maps of Madsen-Swann, except that it is not equivariant in general. As in the symplectic case, it satisfies: 
\begin{cor}
The vector field $v_H$ is tangent to the level sets of the map
$M\overset{\phi}\to Z_n(\g)^*$ given by $\phi(m)(p)=f_n(p)(m)$.
\end{cor}
\begin{proof}
We have to show that $v_{H}(m)\in \ker(T_m\phi:T_mM\to T_{\phi(m)}Z_n(\mathfrak g)^*=Z_n(\mathfrak g)^*)$. We have
\[\big((T_m\phi)(v_H(m))\big)(p)=(df_n(p))(v_H(m))=(\iota_{v_H}df_n(p))(m)=(\cL_{v_H}f_n(p))(m)=0,\]
where the last equation uses the fact that $\cL_{v_H}f_n(p)$ is exact and an exact function is necessarily 0.
\end{proof} 
\begin{remark}
An analogue of this result, where $Z_n(\mathfrak g)$ is substituted by $B_n(\mathfrak g)$ holds in the setting of Proposition \ref{prop:ex-loc-conserved}.
\end{remark}

We close this subsection by showing how a co-momentum map yields elements of the algebra $\cA(v_H)$ from Lemma \ref{lem:module}.
\begin{lemma}For a strictly $H$-preserving infinitesimal action $x\mapsto v_x$ we have:
\begin{compactenum}[(i)]
\item Let $p\in Z_{k}(\g)$ for $k\ge 1$. 
Then $\iota(v_p)\omega \in \cA(v_H)$.

\item Let $(f)$ be a co-momentum map for the $\mathfrak g$-action. Let $p\in Z_{k-1}(\g)$, for $k\ge 2$.
Then $l_k(f_1(x_1),\dots,f_1(x_{k-1}),H) \in \cA(v_H)$.
\end{compactenum}
\begin{proof}
Let us first observe that if $\alpha\in \cC(v_H)$, then $d\alpha\in \cA(v_H)$. In fact $d\alpha$, being exact, is closed. Furthermore $\cL_{v_H}(d\alpha)=d(\cL_{v_H}\alpha)=0$ since $\cL_{v_H}\alpha$ is zero.
\begin{compactenum}[(i)]
\item By Proposition \ref{prop:conserved}, $f_k(p)$ is a globally conserved quantity. As $\iota(v_p)\omega=\pm df_k(p)$ due to Equation \eqref{main}, it is an element of $\cA(v_H)$ because of the preceding observation.

\item By Proposition \ref{conserved-subalgebra}, $l_k(f_1(x_1),\dots,f_1(x_{k-1}),H)$ is conserved. We compute
\begin{align*}
& d(\iota(v_1\wedge\dots \wedge v_{k-1})H)=(-1)^{k-1}(\iota(v_1\wedge\dots \wedge v_{k-1})dH)\\
&=- (\iota(v_1\wedge\dots \wedge v_{k-1}\wedge v_H)\omega)
=-\vs(k)l_k(f_1(x_1),\dots,f_1(x_{k-1}),H),
\end{align*}
so $l_k(f_1(x_1),\dots,f_1(x_{k-1}),H)$ is exact, and in particular closed.
\end{compactenum}
\end{proof}
\end{lemma}

\subsection{The homological point of view}\label{subsec:homol}

In this subsection we rephrase the ``generation of conserved quantities'' via a co-momentum in a homological fashion.

 Let $\mathfrak g$ be a Lie algebra acting on a pre-$n$-plectic manifold $(M,\omega)$, let $H$ be a Hamiltonian $(n{-}1)$-form, and $v_H$ a Hamiltonian vector field of $H$. Assume that 
the action is locally $H$-preserving, i.e. $\cL_{v_x}H$ is closed for all $x\in \g$.
The map $\g\to H^{n-1}_{dR}(M), x\mapsto [\cL_{v_x}H]$ measures how far the action is from being globally $H$-preserving. This map is 0 on $[\g,\g]$ and can thus be defined on $H_1(\mathfrak g)=\g/[\g,\g]$. Furthermore it can be extended to a map on the whole Lie algebra homology.

\begin{prop}\label{lem:A}
For every $k=1,\dots,dim(\g)$ the map
\begin{equation*}
A\colon H_k(\g)\to H_{dR}^{n-k}(M),\;\; [p]\mapsto [\cL_{v_p}H]
\end{equation*}
is well-defined.
\end{prop}

\begin{proof}
 Let $p\in Z_k(\g)$. We first check that $\cL_{v_p}H$ is closed: Putting $v_p=\sum_l v_1^l\wedge...\wedge v_k^l$ one has
$$d\cL_{v_p}H=(-1)^{k+1}\cL_{v_p}dH=-\left(\iota({v_{\partial p}})dH+\sum_l\sum_{i=1}^k(-1)^i \iota(v_{1}^l \wedge \cdots
 \wedge \hat{v}^l_{i} \wedge \cdots \wedge {v}^l_{m})\cL_{v^l_i}dH\right)=0,$$
where the first equality follows from Definition \ref{multidef}, the second from Remark \ref{rem:Lieder} and the last one from $\partial p=0$ and the closeness of $\cL_{v^l_i}H$.

Let $q\in \Lambda^{k+1}\g$. Similarly to above, we write $v_q=\sum_l v_1^l\wedge...\wedge v_{k+1}^l$.
We check that $\cL_{v_{\partial q}}H$ is exact. By the definition of Lie derivative, this follows since 
$$\iota(v_{\partial q})dH=(-1)^{k+1}\cL_{v_q}dH-\sum_{i=1}^{k+1}\sum_l (-1)^{i} \iota( v^l_{1} \wedge \cdots
 \wedge \hat{v}^l_{i} \wedge \cdots \wedge {v}^l_{k+1})\cL_{v^l_i}dH=-d\cL_{v_q}H$$
is exact. Again, here in the first equality we used Remark \ref{rem:Lieder} and in the second that $\cL_{v_i^l}H$ is closed since the action is locally $H$-preserving.
\end{proof}

\begin{remark}\label{rem:A}$ $
\begin{compactenum}[(1)]
\item If the action is globally $H$-preserving, the map $ \g\to H^{n-1}_{dR}(M)$ is zero, but the higher components of $A$ do not necessarily vanish. This is exhibited by Example \ref{local-example}: $\iota(v_a\wedge v_b)dH=-\iota(\frac{\partial}{\partial x}\wedge \frac{\partial}{\partial y})(dx\wedge dy)=-1$
is closed but not exact.
\item If the action is strictly $H$-preserving, then the map $A$ is identically zero.
Indeed, for every $p\in Z_k(\g)$ we have 
$\cL_{v_p}H=0$, as can be seen applying Lemma \ref{tech_lemma} to $H$.
\end{compactenum}
\end{remark}

When a  co-momentum map exists, we can be more explicit:
 \begin{lemma}\label{lem:rewrite} If $(f)$ is a  co-momentum map $(f)$ for the $\mathfrak g$-action, then the map $A$ can be written as follows: for all $p\in Z_k(\g)$,
$$A([p])=-\vs(k)
[\cL_{v_H}f_k(p)].$$
\end{lemma}
\begin{proof}
Let $p\in Z_k(\g)$. We have
$A([p])=[\cL_{v_p}H]=(-1)^{k}[\iota(v_p)\iota_{v_H}\omega]$
using the definition of Lie derivative for multivector fields (see Remark \ref{rem:Lieder}).
We can express this in terms of the co-momentum map using 
$$(-1)^k\iota(v_p)\iota_{v_H}\omega=\iota_{v_H}\iota(v_p)\omega=-\vs(k)\iota_{v_H}df_k(p)=
-\vs(k)(-d\iota_{v_H}f_k(p)+\cL_{v_H}f_k(p))
.$$ Passing to the cohomology class finishes the proof.
\end{proof}

Lemma \ref{lem:rewrite} has several consequences:

\begin{remark}$ $
\begin{compactenum}[(1)]
\item The form $f_k(p)$ is locally conserved if $p\in Z_k(\g)$ and globally conserved if $p\in B_k(\g)$. Hence we recover 
Proposition \ref{prop:loc-conserved} and Proposition \ref{prop:ex-loc-conserved}.

\item There is a canonical injective map $J\colon \frac{\cC_{loc}(v_H)}{\cC(v_H)} \hookrightarrow H_{dR}(M), [\alpha]\mapsto [\cL_{v_H}\alpha]$, as follows immediately from the definitions. The map 
 $A$ factors as
$$H_k(\mathfrak g ){\longrightarrow} \frac{\cC_{loc}^{n-k}(v_H)}{\cC^{n-k}(v_H)}\overset{J}{\longrightarrow} H_{dR}^{n-k}(M)$$
for every $k$, where the first map is induced by $f_k$ multiplied by $-\vs(k)$.
In particular, 
the map $A$ takes values in the subspace $J(\frac{\cC_{loc}(H)}{\cC(H)})$ of $H_{dR}(M)$.

\item If the action is strictly $H$-preserving, by Remark \ref{rem:A} (2), 
 $f_k(p)$ is globally conserved for all $p\in Z_k(\g)$. Hence we recover
 Proposition \ref{prop:conserved}.
\end{compactenum}
\end{remark}

\section{Examples and constructions}\label{sec:exconstr}
In the previous section we proved the existence of conserved quantities in the following set-up: $G$ is a Lie group acting on a pre-$n$-plectic manifold $(M,\omega)$, $H\in \ham{n-1}$ is locally, globally or strictly preserved, and $(f)$ is a co-momentum map . Here we give several constructions that assure the existence of a preserved Hamiltonian $H$ and of a co-momentum map,  in \S \ref{subsec:cons} and \S \ref{subsec:isot}. Further, in  \S \ref{g-plus-r} and \S \ref{subsec:magnetic} we make additional remarks on conserved quantities.

\subsection{Constructing preserved Hamiltonians}\label{subsec:cons}

In this subsection we consider a natural geometric situation in which the above machinery can be applied. We do not assume the existence of an invariant Hamiltonian form here, but we always assume a smooth action $\vartheta: M\times G\to M$ of a
connected Lie group $G$ on a pre-$n$-plectic manifold $(M,\omega)$ such that $\vartheta^*_g(\omega)=\omega$ for all $g\in G$.

\begin{lemma}\label{lem:strictp}
Consider a connected compact Lie group $G$ acting on the pre-$n$-plectic manifold $(M,\omega)$, and let  $\tilde{H}\in \ham{n-1}$ which is \emph{locally} preserved by the action. Then there exists $ {H}\in \ham{n-1}$ which is \emph{strictly} preserved by the action and has the following property: any Hamiltonian vector field of $\tilde{H}$ is also a Hamiltonian vector field of ${H}$. 
\end{lemma}

\begin{proof}
Define ${H}:=\int_G \vartheta^*_g(\tilde H)\mu(dg)$, the average of $\tilde{H}$ using the normalized Haar measure $\mu(dg)$ on $G$. Then $H$ is strictly preserved by the action. Furthermore we have $$dH=\int_G \vartheta^*_g(d\tilde H)\mu(dg) 
=d\tilde H.$$ The last equality holds because 
$\vartheta^*_g(d\tilde H)=d\tilde H$ for all $g\in G$, as a consequence of 
 $d\cL_{v_x}\tilde H=\cL_{v_x}d\tilde H=0$ for all $x\in \g$. Hence the Hamiltonian vector fields of $\tilde{H}$ are exactly the Hamiltonian vector fields of $H$. 
\end{proof}

\begin{prop}\label{prop:vvh} 
Consider a connected compact Lie group $G$
 acting on the pre-$n$-plectic manifold $(M,\omega)$. Let $v$ be a $G$-invariant vector field on $M$ with $\cL_v\omega=0$. Suppose that $H^{n}_{dR}(M)=0$. Then $v$ is the Hamiltonian vector field of some Hamiltonian form 
$H$
which is $G$-invariant, i.e. strictly preserved by the action.

\end{prop}
\begin{proof}
The condition $\cL_v\omega=0$ implies that $\iota_v\omega$ is closed, and by our cohomological assumption it is exact: 
\begin{equation*}\label{eq:tildeH}
\iota_v\omega=-d\tilde{H}
\end{equation*}
for some $\tilde{H}\in \Omega_{Ham}^{n-1}(M)$. Now we average both sides of the above equation.
Notice that $\iota_v\omega$ is $G$-invariant, since $v$ and $\omega$ are, hence the averaged form $H:=\int_G \vartheta^*_g(\tilde H)\mu(dg)$ (where $\mu$ is the normalized Haar measure on $G$) satisfies the same equation: $\iota_v\omega=-d{H}$.
\end{proof}

An interesting special case is the one of volume forms:
\begin{cor}\label{cor:volform}
Consider a connected compact Lie group $G$ acting 
 on $(M,\omega)$,
 where $\omega$ is a volume form. Let $v$ be a $G$-invariant vector field on $M$ which is divergence-free (i.e. $\cL_v\omega=0$). Suppose that $H^{dim(M)-1}_{dR}(M)=0$.
 Then $v$ is the Hamiltonian vector field of some $G$-invariant Hamiltonian form $H$.
\end{cor}

\begin{remark}
If $M$ is compact and simply connected, then $H^{dim(M)-1}_{dR}(M)$ vanishes by Poincar\'{e} duality and the above result is applicable.
\end{remark}

The following statement is a variation of Proposition \ref{prop:vvh},
in which the compactness assumption on $G$ is replaced with the condition $H^{n-1}_{dR}(M)=0$ and which leads to a globally preserved Hamiltonian.

\begin{prop}
Consider a connected Lie group $G$ acting 
 on the pre-$n$-plectic manifold $(M,\omega)$.
 Let $v$ be a $G$-invariant vector field on $M$ with $\cL_v\omega=0$. 
Suppose that $H^{n}_{dR}(M)=0=H^{n-1}_{dR}(M)$. 
 Then $v$ is the Hamiltonian vector field of some Hamiltonian form $H$ which is globally preserved by the action. 
\end{prop}
\begin{proof}
Since $\cL_v\omega=0$ and $H^{n}_{dR}(M)=0$ we have $\iota_v\omega=-d {H}$ 
for some (usually not $G$-invariant) $ {H}\in \ham{n-1}$. 
The form $\iota_v\omega$ is $G$-invariant, since $v$ and $\omega$ are. 
Hence for all $x\in \g$ we have $0=\cL_{v_x}dH=d(\cL_{v_x}H)$. The condition $H^{n-1}_{dR}(M)=0$ implies that 
$\cL_{v_x}H$ is exact.
\end{proof}

\subsection{Induced actions of isotropy subgroups}\label{subsec:isot}

Let $G$ act on a pre-$n$-plectic manifold $(M,\omega)$. In this whole subsection we fix $p\in Z_{k}(\g)\subset \Lambda^k \g$, for some $k\ge 1$ (see Definition \ref{def:cob}). We denote by $G_{p}$ the corresponding isotropy group for the adjoint action of $G$ on $\Lambda^{k}\g$, and by $\g_p$ its Lie algebra. Explicitly, $\g_p=\{x\in \g: [x,p]=0\}$.

\begin{remark}\label{lem:iso}
Let $p\in Z_{k}(\g)$ and $x\in \g$. From Lemma \ref{leibnizg} it follows that $x\wedge p\in Z_{k+1}(\g)$ if{f}
$x\in \g_p$.
\end{remark}

\begin{lemma}\label{lem:closedinv}
The form $\iota(v_p)\omega \in \Omega^{n+1-k}(M)$ is closed, and invariant under
the action of $G_p^0$, the connected component of the identity in $G_p$.
\end{lemma}
\begin{proof}
The equality $d(\iota(v_p)\omega)=0$ follows upon applying Lemma \ref{tech_lemma}.
The invariance holds since for every $y\in \g_p$ we have
$\cL_{v_y}\iota(v_p)\omega=\iota({[v_y,v_p]})\omega+\iota(v_p)\cL_{v_y}\omega=0$, where the bracket is defined analogously to Lemma \ref{leibnizg} later on.
\end{proof}

Now assume there is a co-momentum map $(f) \colon \g \to L_{\infty}(M,\omega)$.

\begin{prop}\label{prop:indmomap}
A co-momentum map for the action of $G_p^0$ on $(M,\iota(v_p)\omega)$ is given by
$(f^p) \colon \g_p \to L_{\infty}(M,\iota(v_p)\omega)$ with components ($j=1,\dots,n-k)$:
\begin{align*}
f^p_j \colon \Lambda^j\g_p &\to\Omega^{n-k-j}(M),\\ q&\mapsto -\vs(k)f_{j+k}(q\wedge p).
\end{align*}
Furthermore, if the co-momentum map $(f)$ is $G$-equivariant, then $(f^p)$ is $G_p^0$-equivariant.
\end{prop}

\begin{proof} 
We first show that $(f^p)$ is a co-momentum map.
Let $q\in \Lambda^j\g_p$. We have
$$\vs(k+j)\iota_{v_q}(\iota(v_p)\omega)=\vs(k+j) \iota(v_{p\wedge q})\omega= -f_{k+j-1}(\partial(p\wedge q))-d(f_{k+j}(p\wedge q))$$
using in the second equality that $(f)$ is a co-momentum map (see eq. \eqref{main}
). 
From Lemma \ref{leibnizg} we obtain $\partial(p\wedge q)= (-1)^k p\wedge \partial(q)=(-1)^{kj} \partial(q)\wedge p$
since $p\in Z_{k}(\g)$ and $q\in \wedge^j\g_p$. Using $\vs(k+j)=\vs(k)
\vs(j)(-1)^{kj+1}$ we hence obtain
$$\vs(j)\iota_{v_q}(\iota(v_p)\omega)=-f^p_{j-1}(\partial(q))-df^p_j(q).$$
For the equivariance statement, notice that for all $y\in \g_p$:
$$\cL_{v_y}(f^p_j(q))=\cL_{v_y}(f_{k+j}(p\wedge q))=f_{k+j}({[y,p\wedge q]})=f_{k+j}({p\wedge [y,q]})=f^p_j([y,q]),$$
where we used the equivariance of $(f)$ in the second equality.
\end{proof}

\begin{remark}
The existence of a co-momentum map $(f)$ implies that $\iota(v_p)\omega$ 
is exact with primitive $-\vs(k)f_k(p)$, 
 by Equation \eqref{main}.
Assume further that $f_k$ is $G$-equivariant. Then this primitive is $G_p^0$-invariant, for $\cL_{v_y} f_k(p)=f_k([y,p])=0$ for all $y\in \g_p$. Hence, by \cite[Lemma 8.1]{FRZ}, an equivariant co-momentum map
for the action of $G_p^0$ on $(M,\iota(v_p)\omega)$ is given by ($j=1,\dots,n-k)$:

\begin{align*}
\Lambda^j\g_p &\to\Omega^{n-k-j}(M),\\ q&\mapsto (-1)^{k}\iota(v_q)(f_{k}(p)).
\end{align*}

Notice that this co-momentum map may differ from the one given in Proposition \ref{prop:indmomap}.
\end{remark}

Finally, we consider Hamiltonian forms.

\begin{prop}\label{ipHinv}
Let $H\in \ham{n-1}$ be $G$-invariant, then $\iota(v_p)H$ is $G_p^0$-invariant and it is a Hamiltonian form with respect to $\iota(v_p)\omega$ with Hamiltonian vector field $v_H$.
\end{prop}
\begin{proof}
The $G_p^0$-invariance of $\iota(v_p)H$ is shown exactly as in Lemma \ref{lem:closedinv}. For the second statement, using Lemma \ref{tech_lemma} we compute $d(\iota(v_p)H)=(-1)^k\iota(v_p)dH=\iota_{X_H}(\iota(v_p)\omega).$ 
\end{proof}

\begin{remark}
Consider the case $k=n-1$. Then $\iota(v_p)\omega$ is a 2-form, and from Proposition \ref{ipHinv} we recover the fact that $\iota(v_p)H$ is a conserved quantity (a special case of Proposition \ref{prop:Hcons} later on).
\end{remark}

\subsection{Co-momentum maps for $\g\oplus \RR$} \label{g-plus-r}

We extend the results of \S \ref{sec:strict} under a non-degeneracy assumption for $\omega$. 
We assume that $(M,\omega)$ is an $n$-plectic manifold, $H\in\ham{n-1}$ and that $(f):\mathfrak g\to \li(M,\omega)$ is a co-momentum for a strictly $H$-preserving infinitesimal action $\mathfrak g\to \mathfrak X(M,\omega),x\mapsto v_x$.

By Lemma \ref{conserved-properties2} (i) the generators of the action commute with the Hamiltonian vector field $v_H$ of $H$, so the infinitesimal $\g$-action on $M$ extends to an action of the direct sum Lie algebra $\tilde{\g}:=\g\oplus \langle c \rangle_\RR$, by means of $c \mapsto v_H$. Notice that $\Lambda^k\tilde{\g}=\Lambda^k\g\oplus (\Lambda^{k-1}\g\otimes \langle c \rangle_\RR)$.

In the sequel we will make use of the following Lemma several times. Recall that the differential $\partial$ was defined in Def. \ref{def:differential}.
\begin{lemma} \label{leibnizg}
Let $p\in \Lambda^k\g$ and $q\in\Lambda^l\g$. Then
$$\partial(p\wedge q)=\partial(p)\wedge q+ (-1)^k p\wedge \partial(q)+(-1)^k[p,q],$$
where $[x_1\wedge...\wedge x_k,y_1\wedge ...\wedge y_l]=\sum (-1)^{i+j}[x_i,y_j]\wedge x_1 \wedge ...\wedge \hat x_i\wedge ...\wedge x_k \wedge y_1 \wedge ...\wedge \hat y_j\wedge ...\wedge y_l$. 
\end{lemma}
\begin{proof}
It is sufficient to prove the assertion for monomials
$p= x_1\wedge\ldots\wedge x_k$ and $q=x_{k+1}\wedge\ldots\wedge x_{k+l}$.
In that case $\partial(p\wedge q)$ is given by a 
sum over indices $i,j$ with ${1 \leq i < j \leq k+l}$. Splitting it into sums over ${ i < j \leq k}$, ${k< i < j }$ and $ i \leq k < j$ proves the assertion.
\end{proof}

\begin{remark}
The bracket $[\cdot,\cdot]:\Lambda^\bullet\mathfrak g\times \Lambda^\bullet\mathfrak g\to \Lambda^\bullet\mathfrak g$ defined above turns $\Lambda^\bullet\mathfrak g$ into a Gerstenhaber algebra.
\end{remark}

\begin{lemma}\label{lem:extmomap}
There is a canonical extension of $(f)$ to a co-momentum map $(\tilde{f})$ for the $\tilde{\g}$-action, determined by 
\begin{equation}\label{eq:ext} 
\tilde{f}_k(x_1,\dots,x_{k-1}, c)= \vs(k)
\iota(v_{x_1}\wedge\dots \wedge v_{x_{k-1}})H
\end{equation} 
for all $k\ge 1$ and $x_1,\dots,x_{k-1}\in \g$.
\end{lemma}
\begin{proof} We have to check that Equation \eqref{main} is satisfied. Without loss of generality assume $p=x_1 \wedge \dots \wedge x_{k-1}\in \wedge^{k-1}\g$, and notice
that $[x_i,c]=0$ for all $i$ implies that $\partial (p\otimes c)=(\partial p)\otimes c$, by Lemma \ref{leibnizg}. 
Using the definition of $\tilde{f}_{k}$, Equation \eqref{main} applied to $p\otimes c$ reads 
$$
-\vs(k-1)\iota(v_{\partial p})H={\vs(k)}d\iota(v_p)H+
\vs(k)(-1)^{k-1}\iota(v_p)\iota_{v_H}\omega.
$$

Using $\iota_{v_H}\omega=-dH$ we see that this equation is satisfied by Lemma \ref{tech_lemma}, since $H$ is $\mathfrak g$-invariant and using the identity $\vs(k)\vs(k-1)=(-1)^k$.
\end{proof}

Notice that, even when $(f)$ is equivariant, $(\tilde{f})$ is not equivariant in general. For instance,
$\cL_{v_H}f_1(x)$ is usually different from $\tilde{f}_1([c,x])=\tilde{f}_1(0)=0$. Further, it can usually not be made equivariant by an averaging procedure since the group $G\times \RR$ integrating $\tilde{\g}$ is non-compact. 

If $(\tilde{f})$ is equivariant, one has strong consequences: $\cL_{v_H} f_k(x_1,\dots,x_{k})=0$, and
in particular $f_k(x_1,\dots,x_{k})$ is a strictly conserved quantity for all $x_1\wedge\dots\wedge x_{k}\in \wedge^k\g$.\\
 
 If we assume that $\cL_{v_H}H=0$, then the $\tilde{\g}$-action strictly preserves the Hamiltonian $H$, hence we can apply Proposition \ref{prop:conserved} to the $\tilde{\g}$-action and obtain globally conserved quantities for $v_H$ for all
 elements of $Z_k({\tilde{\g}})$. As the latter is isomorphic to $Z_k(\g)\oplus (Z_{k-1}(\g)\otimes \langle c\rangle_\RR)$,
these globally conserved quantities are those we already know from 
 Proposition \ref{prop:conserved}, plus
 those arising from $Z_{k-1}(\g)\otimes \langle c\rangle_\RR$. Somewhat surprisingly, it turns out that the latter are globally conserved quantities even without the assumption $\cL_{v_H}H=0$. {This fact is not predicted by Prop. \ref{prop:loc-conserved}, which only ensures the existence of locally conserved quantities.}

\begin{prop}\label{prop:Hcons} 
Assume that $(M,\omega)$ is an $n$-plectic manifold, $H\in\ham{n-1}$ and that $(f):\mathfrak g\to \li(M,\omega)$ is the co-momentum of a strictly $H$-preserving Lie algebra action. The $\tilde f_k(p\otimes c)$ as in eq. \eqref{eq:ext} is a globally conserved quantity, for all $p\in Z_{k-1}(\mathfrak g)$.
\end{prop}
\begin{proof}
Because of Cartan's formula it suffices to show that $\iota_{v_H} d \tilde f_k(p\otimes c)$ is exact. We will show that it actually vanishes. By Lemma \ref{lem:extmomap} and using eq. \eqref{main} we get
\[
\iota_{v_H} d \tilde f_k(p\otimes c)=\iota_{v_H} \left( -\tilde f_{k-1}(\partial(p\otimes c)) -\vs(k)\iota(v_{p\otimes c})\omega\right)
\]
Applying Lemma \ref{leibnizg} to the Lie algebra $\tilde{\g}$ we see that $\partial(p\otimes c)=0$. Further $v_{c}={v_H}$, so by the skew-symmetry of $\omega$ we get $\iota_{v_H} \iota(v_{p\otimes c})\omega=0$, which finishes the proof.
\end{proof}

\subsection{Multisymplectic analogue of magnetic terms}\label{subsec:magnetic}

In this subsection we explain how to generalize the well-known magnetic term from symplectic geometry to the multisymplectic situation (compare  \cite[\S 7]{CIdL99} for this construction) and provide the example announced in \S \ref{sec:strict}. 
\begin{constr}
Let $N$ be a manifold and $c$ a closed $k+1$-form on $N$. Denoting the canonical projection from $\Lambda^kT^*N\to N$ by $\pi$ and the canonical $k$-form on $\Lambda^kT^*N$ by $-\theta$, the $k+1$-form $\omega=d\theta+\pi^*c$ is always $k$-plectic i.e. non-degenerate and closed on $M=\Lambda^kT^*N$. The form $\pi^*c$ is called \emph{magnetic term}.
\end{constr}

\begin{prop}
Let $k\geq 1$  and $N$ be a manifold, $b\in \Omega^k(N)$  and $w$ a vector field on $N$, such that $\cL_wb=da$ for some $a \in \Omega^{k-1}(N)$ (i.e. $b$ is globally conserved by $w$). Denote the canonical lift of $w$ to $M=\Lambda^kT^*N$ by $w^h$. Then $w^h$ is a Hamiltonian vector field on $(M,\omega= d\theta+\pi^*db)$, with the following Hamiltonian $(k{-}1)$-form:
\begin{equation}\label{eq:H}
H=-\pi^*a+\iota_{w^h}(\theta+\pi^*b).
\end{equation}
\end{prop}

\begin{proof}
Upon observing $\iota_{w^h}(\pi^*b)=\pi^*(\iota_wb)$ and consequently $\cL_{w^h}(\pi^*b)=\pi^*(\cL_wb)$, we have:
\begin{align*}
dH&=-\pi^*\cL_wb+d(\iota_{w^h}\theta)+\pi^*(d\iota_wb)\\
&=-\pi^*da-\iota_{w^h}d\theta+\cL_{w^h}\theta-\pi^*(\iota_wdb)+\pi^*\cL_wb\\
&=-\pi^*da-\iota_{w^h}d\theta-\pi^*(\iota_wdb)+\pi^*da=-\iota_{w^h}(d\theta+\pi^*db)\\
&=-\iota_{w^h}\omega,
\end{align*}
where in the third equality we used $\cL_{w^h}\theta=0$, as in the symplectic case.
\end{proof}

\begin{remark}\label{rem:comom}
If $N\times G\to N$ is a right action and $b$ a $G$-invariant $k$-form on $N$, then the 
 $k$-plectic form
 $\omega=d(\theta+\pi^*b)$ on $M=\Lambda^kT^*N$ has a $G$-invariant potential. This assures the existence of a co-momentum map (see \cite[\S 8.1]{FRZ}), whose  first component 
$f_1\colon \g\to \ham{k-1}$ is given by $f_1(x)=\iota_{v_x^h}(\theta+\pi^*b)$.
\end{remark}

In \S \ref{sec:strict} we announced an example of a strictly $H$-preserving action on a multisymplectic manifold admitting a co-momentum $(f)$ such that, for some boundary $x\in B_1(\g)$, $f_1(x)$ is a globally conserved quantity that is \emph{not strictly conserved}. We now provide this example.
\begin{ep} \label{magnetic:ep}  
Let $N\times G\to N$ be a right action and assume that, in the set-up of the preceding proposition, the vector field $w$ and the forms $a$ and $b$ be $G$-invariant. Then $H$ (see eq. \eqref{eq:H}) is invariant under the induced $G$-action on $M=\Lambda^kT^*N$. 

Assuming furthermore $db=0$, we can choose $a:=\iota_wb$. Specialize to $k=2$ and $N=G$ with the action by 
$N \times G \to N, (n,g)\mapsto g^{-1}{ \cdot} n$.  Thus we have: 

\begin{itemize}
\item for $x\in \mathfrak{g}=T_eG$ and $g\in G$ is $v_x(g)=-(r_g)_*(x)$, where $r_g(h)=h{\cdot} g$ for $h\in G$. In particular,  the generators $v_x$ of the action are right-invariant vector fields on 
$G$.
\item $w$ is  a left-invariant vector field, i.e. it exists $\tilde{w}\in \g$ such that for $g\in G$, $w(g)=(l_g)_*(\tilde{w})$,
where $l_g(h)=g{\cdot} h$ for $h\in G$,
\item $b$  is a closed left-invariant $2$-form, i.e. it exists a $\tilde{b}\in \Lambda^2\g^*$ which is closed under the Chevalley-Eilenberg differential (the dual of the Lie algebra homology differential) and 
$b(g)=(l_{g^{-1}})^*(\tilde{b}) =((l_{g^{-1}})_*)^*(\tilde{b})$ for all $g \in G$.
\end{itemize}

Denote by $(f)$ the co-momentum map recalled in Remark \ref{rem:comom}. For any $x\in \g$ we compute
$$\cL_{w^h}f_1(x)=\iota_{v_x^h}\pi^*(\cL_wb)=-\pi^*(d(\iota_{v_x}a)),$$
so $f_1(x)$ being a strictly conserved quantity is equivalent to
$\iota_{v_x}a$ being a constant function on $N$.
Evaluating the function $\iota_{v_x}a=b(w,v_x)$ at $g\in N=G$ one obtains 
\begin{equation}\label{eq:fct}
-\tilde{b}(\tilde{w},Ad_{g^{-1}}(x)).
\end{equation}

It is clear that the function \eqref{eq:fct} is not constant in general. For instance, take $G=SL_2(\RR)$. A basis for $\g:=\mathfrak{sl}_2(\RR)$ is
\[ \begin{array}{ccc}
h=\left( \begin{array}{cc} 1 & 0 \\ 0 & -1 \end{array} \right), &
e=\left( \begin{array}{cc} 0 & 1 \\ 0 & 0 \end{array} \right), &
f=\left( \begin{array}{cc} 0 & 0 \\ 1 & 0 \end{array} \right),
\end{array} \]
and $[h,e]=2e, [h,f]=-2f, [e,f]=h$. So notably all elements of $\g$ are boundaries. The form $\tilde{b}:=e^*\wedge f^*\in \Lambda^2\g^*$ is closed (actually exact) with respect to the Chevalley-Eilenberg differential. Taking $\tilde{w}:=f$ and $x:=h\in \g=B_1(\g)$ one computes that the function \eqref{eq:fct} attains the value $2 \beta \delta$ at 
$g=\left( \begin{array}{cc} \alpha & \beta \\ \gamma & \delta \end{array} \right)\in G$, hence it is not a constant function on $G$. 
We conclude that for this choice of $x \in B_1(\g)$, the form  $f_1(x)$ is  not strictly conserved.
\end{ep}

\section{Applications of conserved quantities}\label{sec:applications}

In \S \ref{sec:conshomo} we saw that many conserved quantities exist on pre-$n$-plectic manifolds endowed with a co-momentum map.
In this section we show some geometric consequences of the existence of conserved quantities on a manifold $M$, by looking at maps from a compact oriented manifold $\Sigma$ into $M$. In most of our statements $M$ does not need any additional geometric structure (but we specialize to the pre-$n$-plectic case e.g. in Proposition \ref{prop:transcons}). 
In \S \ref{sec:general} we consider conserved quantities whose degree, as differential forms on $M$, equals $\dim(\Sigma)$. In \S \ref{sec:transgr} we extend some of the results to arbitrary degrees.

\subsection{A general Kelvin's circulation theorem}\label{sec:general}

Let $M$ be a manifold and $v\in\mathfrak X(M)$ a vector field. Let $\Sigma$ be a compact, oriented $d$-dimensional manifold, and $\sigma_0\colon \Sigma\to M$ a smooth map. We view $\Sigma$ as a ``membrane'' in $M$, which evolves under the flow of the vector field, and want to find quantities which are unchanged under the evolution. 
The following proposition can be viewed as a general version of Kelvin's circulation theorem, as we explain in Remark \ref{rem:Kelvin} below.

\begin{prop}\label{prop:unchanged}
Let $\Sigma$ be a compact, oriented $d$-dimensional manifold (possibly with boundary), $v$ a vector field on $M$ with flow $\phi_t$,
and $\sigma_0\colon \Sigma\to M$ a smooth map. Consider $\sigma_t:=\phi_t\circ \sigma_0\colon \Sigma\to M$.
If $\alpha\in \Omega^d(M)$ is a differential form, 
then the number
$$\int_{\Sigma}(\sigma_t)^*\alpha$$ is independent of the time parameter $t$ if one of the following conditions holds:
\begin{compactenum}[(i)]
	\item $\alpha$ is strictly conserved by $v$,
	\item $\alpha$ is globally conserved by $v$ and $\Sigma$ has no boundary,
	\item $\alpha$ is locally conserved by $v$ and there exists a compact, oriented manifold with boundary $N$ such that $\Sigma=\partial N$ and a map $\tilde \sigma_0:N\to M$ with $\tilde\sigma_0|_{\partial N}=\sigma_0$.
\end{compactenum}

\end{prop}
\begin{remark}
Since $\Sigma$ is compact, there exists an $\varepsilon=\varepsilon(\sigma_0)>0$ such that $\phi_t$ is defined at least on $(-\varepsilon,\varepsilon)\times\sigma_0(\Sigma)\subset \mathbb R\times M$. Obviously in (i) and (ii) we can consider $|t|<\varepsilon=\varepsilon(\sigma_0)$. Mutatis mutandis we consider only $|t|<\varepsilon(\tilde \sigma_0)$ in case (iii). Notice that if $dim(M)=d$, or more generally
if $\alpha$ is closed, then $\alpha$ is globally conserved.
\end{remark}

\begin{proof}
We only prove that condition (ii) suffices, for the other implications follow analogously. The diffeomorphisms $\phi_t$ satisfy $\frac{d}{dt}(\phi_t^*\alpha)=\phi_t^*(\cL_{v}\alpha)$.
Precomposing with the pullback $(\sigma_0)^*$ we obtain
\begin{equation}
\label{eq:ananondiff}
\frac{d}{dt}(\sigma_t^*\alpha)=\sigma_t^*(\cL_{v}\alpha).
\end{equation}
Hence by compactness of $\Sigma$,
\begin{equation*}
\frac{d}{dt}\int_{\Sigma}\sigma_t^*\alpha=\int_{\Sigma}
\sigma_t^*(\cL_{v}\alpha)=\int_{\Sigma}
d(\sigma_t^*\gamma)=0,
\end{equation*}
where in the first equality we used Equation \eqref{eq:ananondiff}, in the second that $\cL_{v}\alpha=d\gamma$ for some form $\gamma$, and in the last one Stokes theorem.
\end{proof}

\begin{remark}
The sufficiency of Condition (ii) in Proposition \ref{prop:unchanged} is not surprising. By assumption $v$ preserves $\alpha$ up to an exact form, and by Stokes' theorem
the contribution given by exact forms vanishes upon integration over $\Sigma$.
\end{remark}

The following statement addresses a variation of condition (iii) in Proposition \ref{prop:unchanged}.
\begin{prop}\label{prop:til}
Let $\Sigma$ be a compact, oriented manifold without boundary of dimension $d$, $v$ a vector field on $M$ with flow $\phi_t$. If $\alpha\in \Omega^d(M)$ is locally conserved, then 
 for every fixed time $t$, one obtains a well-defined map 
$$F_t \colon [\Sigma, M]\to \RR,\;\;\; [\sigma_0]\mapsto \int_{\Sigma}(\sigma_t)^*\alpha-\int_{\Sigma}(\sigma_0)^*\alpha.$$
Here $[\Sigma, M]$ denotes the set of smooth homotopy classes of maps from $\Sigma$ to $M$, $\sigma_0:\Sigma\to M$ denotes a smooth map and $\sigma_t:=\phi_t\circ \sigma_0\colon \Sigma\to M$. 

Further, the dependence on $t$ is linear: $F_t[\sigma_0]=t\cdot c([\sigma_0])$ where $c([\sigma_0]):=\int_{\Sigma}(\sigma_0)^*(\cL_v\alpha)$.
\end{prop} 

\begin{proof}
We have
$$\int_{\Sigma}(\sigma_t)^*\alpha-\int_{\Sigma}(\sigma_0)^*\alpha=
\int_0^t\left[\frac{d}{ds}\int_{\Sigma}(\sigma_s)^*\alpha\right]ds=
\int_0^t\left[\int_{\Sigma}\sigma_s^*(\cL_v\alpha)\right]ds$$
where the last equality is obtained as in the proof of Proposition \ref{prop:unchanged}.
Now recall that $\cL_v\alpha$ is a closed form on $M$. Hence by Stokes' theorem
the term in the square bracket depends only on the homotopy class of $\sigma_s$, which agrees with the homotopy class of $\sigma_0$ since $\sigma_s=\phi_s\circ \sigma_0$.
We conclude that the above expression equals $t\cdot c([\sigma_0])$.
\end{proof}

We present an example for Proposition \ref{prop:unchanged} and Proposition \ref{prop:til}.
\begin{ep}
Let $M$ be a manifold, $v$ a vector field, and $\alpha\in \Omega^d(M)$.
Take a map $\sigma_0\colon S^d\to M$ defined on the $d$-dimensional sphere, denote by $\sigma_t$ the composition of $\sigma_0$ with the time $t$ flow $\phi_t$ of $v$. The number
$\int_{S^d}(\sigma_t)^*\alpha$ is independent of the time parameter $t$ if the following occurs: either i) $\cL_v\alpha$ is exact, or ii) $\cL_v\alpha$ is closed and $\sigma_0$ is homotopy equivalent to a constant map. This follows from Proposition \ref{prop:unchanged} (ii) and (iii).

Further, assuming that $\cL_v\alpha$ is closed, one obtains a well-defined group homomorphism
$$\pi_d(M,{x})\to \RR,\;\;\; [\sigma_0]\mapsto \int_{S^d}(\sigma_t)^*\alpha-\int_{S^d}(\sigma_0)^*\alpha$$ defined on the $d$-th homotopy group of $M$ based at some point $x$, and where the dependence on $t$ is linear. This follows from Proposition \ref{prop:til} and the following argument to show that the group homomorphism property. We denote the group multiplication of $\pi_d(M,x)$ by $*$. It is given by the following composition, where $p$ denotes a distinguished point on the sphere:
\[
f*g:(S^d,p)\to (S^d/ S^{d-1},p)=(S^d\vee S^d,p)\overset{(f\vee g)}\longrightarrow (M,x)
\]
Choosing appropriate representatives of the {respective homotopy classes} we may assume that $f,g$ and $f*g$ are smooth. 
Then for $\alpha\in\Omega^d(M)$ we calculate:
\[\int_{S^d}(f*g)^*\alpha=
\int_{S^d\backslash \{p\}\sqcup S^d\backslash \{p\}}(f\vee g)^*\alpha
=\int_{S^d}f^*\alpha+\int_{S^d}g^*\alpha.\] \end{ep}

\begin{remark}[{Kelvin circulation theorem}]\label{rem:Kelvin}
A variant of Proposition \ref{prop:unchanged} (ii) for a time-dependent vector field $v^t$ and time-dependent differential form $\alpha^t\in \Omega^d(M)$ is the following: if 
$\cL_{v^t}\alpha^t+\frac{d}{dt}\alpha^t$ is exact and $\Sigma$ has no boundary, 
then the number
$$\int_{\Sigma}(\sigma_t)^*(\alpha^t)$$ is independent of the time parameter $t$.

We mention this because \emph{the Kelvin circulation theorem in fluid mechanics can be understood as a special case of the above}. Let $v^t= \sum_iv^t_i\partial_{x_i}$ be a time-dependent vector field on $\RR^3$, and use the standard metric on $\RR^3$ to obtain from $v^t$ the 1-form $\alpha^t= \sum_iv^t_idx_i$. One computes 
$\iota_{v^t}d\alpha^t=\sum_{i,k}v^t_k\frac{\partial v^t_i}{\partial x_k} dx_i-\frac{1}{2}d\sum_i(v_i^t)^2$, so that $\cL_{v^t}\alpha^t+\frac{d}{dt}\alpha^t$ is exact if and only if
\begin{equation}\label{eq:lhseul}
\sum_{i,k}v^t_k\frac{\partial v^t_i}{\partial x_k} dx_i+\sum_i\left(\frac{d}{dt}v_i^t\right)dx_i
\end{equation}
 is exact. By the above, it then follows that $\int_{\Sigma}(\sigma_t)^*\alpha^t$ is independent of $t$.

Upon rewriting the exactness of \eqref{eq:lhseul} as $(v^t\cdot \nabla)v^t+\frac{\partial u}{\partial t}=-\nabla w$, with $\nabla$ the usual gradient in $\mathbb R^3$, we recognise the first of the isentropic Euler equations (see e.g.  \cite[page 15]{Fluid}). It is well-known that this equation implies the classical Kelvin circulation theorem (\cite[page 21]{Fluid}) which is exactly the time-independence of $\int_{\Sigma}(\sigma^*_t)(\alpha^t)$ in this case.\\

Under symmetry assumptions, our methods lead to other conserved quantities, as we now explain. 
In fluid mechanics, the time-dependent vector field $v^t$ is divergence free, i.e. $\cL_{v^t}\omega=0$ for all $t$, where $\omega=dx_1\wedge dx_2\wedge dx_3$ is the standard volume form on $\RR^3$. In the presence of a compact group of symmetries -- i.e. of an action of a compact Lie group $G$ on $\RR^3$ preserving $\omega$ and $v^t$ -- it turns out by Corollary \ref{cor:volform} that, for every fixed value of $t$, the vector field $v^t$ is the Hamiltonian vector field of a time dependent $G$-invariant Hamiltonian $1$-form of $(\mathbb R^3, \omega)$. Further, $\omega$ is exact with $G$-invariant primitive, so that the action of $G$ on $(\RR^3,\omega)$ admits a  co-momentum map \cite[\S 8.1]{FRZ}. 
The latter, by virtue of Proposition \ref{prop:conserved}, delivers further (time-independent) globally conserved quantities for $v^t$, for each value of $t$. One can then apply the above variant of Proposition \ref{prop:unchanged} (ii) to the time-dependent vector field $v^t$ and to the newly obtained globally conserved quantities.
\end{remark}

\subsection{Transgression of conserved quantities}\label{sec:transgr}
Proposition \ref{prop:unchanged} (ii) fits in the following framework. Let $\Sigma$ be a compact, oriented manifold (without boundary) and $M$ a manifold. Given a vector field $v$ on $M$, there is a naturally associated vector field $v^{\ell}$ on 
$M^{\Sigma}=C^\infty(\Sigma,M)$, 
 the space of smooth maps from $\Sigma$ to $M$. It is
 given as follows: $$v^{\ell}|_{\sigma}=\sigma^*v\in \Gamma(\sigma^*TM)=T_{\sigma}M^{\Sigma},$$ for all $\sigma\in M^{\Sigma}$. Notice that, denoting by $\phi_t$ the flow of $v$ on $M$,
the flow of $v^{\ell}$ maps $\sigma\in M^{\Sigma}$ to 
$\phi_t\circ \sigma\in M^{\Sigma}$. Similarly, associated to a differential form on $M$ there is a differential form on $M^{\Sigma}$ of lower degree. It is defined by the transgression map
\[
\ell:=\int_{\Sigma}\circ\; ev^* \colon \Omega^{\bullet}(M) \to \Omega^{\bullet-s}(M^{\Sigma})
\] 
where 
$ev\colon \Sigma \times M^{\Sigma}\to M$
is the evaluation map and 
$\int_{\Sigma}$ denotes the integration along the fiber (cf. eg. \cite[Cap. VI.4]{audin2004torus}) of the projection $\Sigma \times M^{\Sigma}\to M^{\Sigma}$.

\begin{prop}\label{prop:transconsnoG}
Let $\Sigma$ be a compact, oriented manifold (without boundary) of dimension $d$ and let $v$ be a vector field on $M$.
If $\alpha\in \Omega^k(M)$ is globally (resp. locally resp. strictly) conserved by $v$ then $\alpha^{\ell}\in \Omega^{k-d}(M^{\Sigma})$ is globally (resp. locally resp. strictly) conserved by $v^{\ell}$.

\end{prop}
\begin{proof}
The transgression map $\ell$ commutes with de Rham differentials, and furthermore we have $(\iota_v\alpha)^{\ell}=
\iota_{v^{\ell}}\alpha^{\ell}$. Therefore it commutes with Lie derivatives in the following sense: $\cL_{v^\ell}\alpha^\ell=(\cL_v\alpha)^\ell$. Assume $\alpha$ is globally conserved. We have to show that $\cL_{v^{\ell}}\alpha^{\ell}$ is an exact form. 
Since $\alpha\in \Omega^k(M)$ is a globally conserved quantity, there is a $\gamma\in \Omega^{k-1}(M)$ with $\cL_{v}\alpha=d\gamma$. Hence
$$\cL_{v^{\ell}}\alpha^{\ell}=(\cL_v\alpha)^{\ell}=
(d\gamma)^{\ell}=d\gamma^{\ell}.$$
The other cases follow similarly.
\end{proof}

\begin{remark}
When $k=d$, 
Proposition \ref{prop:transconsnoG} recovers Proposition \ref{prop:unchanged} (ii). Indeed, let $\alpha\in \Omega^d(M)$. By Proposition \ref{prop:transconsnoG}
the function $\alpha^{\ell}$ on $M^{\Sigma}$ is invariant under the flow of $v^{\ell}$. The latter maps a point $\sigma\in M^{\Sigma}$ to 
$\phi_t\circ \sigma$, where $\phi_t$ denotes the flow of $v$ on $M$.
Finally, for all $\sigma \in M^{\Sigma}$ we have
$$\alpha^{\ell}|_{\sigma}=\big(
(\int_{\Sigma}\circ\; ev^*)(\alpha)\big)|_{\sigma}=\int_{\Sigma} (ev|_{\Sigma\times \{\sigma\}})^*\alpha=
\int_{\Sigma}\sigma^*\alpha,$$
where in the second equality we used that $ev|_{\Sigma\times \{\sigma\}}=\sigma$. 

\end{remark}

Now we specialize to a pre-$n$-plectic manifold $(M,\omega)$ together with a
vector field $v_H$ which is Hamiltonian for some $H\in\ham{n-1}$.
Notice that 
 $(v_H)^{\ell}$ is a Hamiltonian vector field of $H^{\ell}$ on $(M^{\Sigma},\omega^{\ell})$, as follows from 
$$\iota_{(v_H)^{\ell}}\omega^{\ell}=(\iota_{v_H}\omega)^{\ell}
=(-dH)^{\ell}=-dH^{\ell}.$$

A special case of Proposition \ref{prop:transconsnoG} reads:
\begin{prop}\label{prop:transcons} 
Consider a pre-$n$-plectic manifold $(M,\omega)$ together with a
vector field $v_H$ which is Hamiltonian for some $H\in\ham{n-1}$.
Let $\Sigma$ be a compact, oriented manifold (without boundary) of dimension $d$. 
If $\alpha\in \Omega^k(M)$ is a globally conserved quantity for $v_H$, then 
$$\alpha^{\ell}\in \Omega^{k-d}(M^{\Sigma})$$
 is a globally conserved quantity for $(v_H)^{\ell}$, i.e.
 $\cL_{(v_H)^{\ell}}\alpha^{\ell}$ is an exact form.
\end{prop}

\begin{remark}
Consider a $G$-action on $(M,\omega)$ for which $H$ is strictly preserved.
The $G$-action on $M$ gives rise to a $G$-action on
$M^{\Sigma}=C^\infty(\Sigma,M)$, simply given by $(g\cdot \sigma)(p):=g\cdot \sigma(p)$ for all $\sigma\in M^{\Sigma}$ and $p\in \Sigma$. It can be checked that for a 
infinitesimal generator $v_x$ of the action on $M$ ($x\in \g$),
the corresponding infinitesimal generator of the action on $\Sigma$ is $(v_x)^{\ell}$. Hence the lifted $G$-action preserves $\omega^{\ell}$ and $H^{\ell}$.
In \cite[\S 11]{FRZ} and \cite[\S 6]{FLRZ} it is shown that 
a co-momentum map for the action of 
$G$ on $(M,\omega)$
transgresses to a co-momentum map for the action on 
$(M^{\Sigma},\omega^{\ell})$. This is consistent with the fact, that firstly, certain components of co-momentum maps  are globally conserved quantities (Proposition \ref{prop:conserved}) and, secondly, globally conserved quantities transgress to give globally conserved quantities (Proposition \ref{prop:transcons}).
\end{remark}

 \bibliographystyle{habbrv} 
\bibliography{Conserved}
\end{document}